\documentclass[12pt]{amsart}
\usepackage{amssymb}
\usepackage{amsmath}
\usepackage{bbm,float}
\usepackage{psfrag}
\usepackage{hyperref}

\numberwithin{equation}{section}
\newcommand{\be}{\begin{equation}}
\newcommand{\ee}{\end{equation}}
\newcommand{\ds}{\displaystyle}

\newtheorem{thm}{Theorem}
\newtheorem{cor}[thm]{Corollary}
\newtheorem{lem}{Lemma}
\newtheorem{prop}{Proposition}
\newtheorem{defn}{Definition}
\newtheorem{rem}{Remark}

\makeatletter
\def\Ddots{\mathinner{\mkern1mu\raise\p@
\vbox{\kern7\p@\hbox{.}}\mkern2mu
\raise4\p@\hbox{.}\mkern2mu\raise7\p@\hbox{.}\mkern1mu}}
\makeatother

\title{Monotone triangles and 312 Pattern Avoidance}
\author[A. Ayyer]{Arvind Ayyer}
\address{Arvind Ayyer\\
Institut de Physique Th\'eorique\\
IPhT, CEA Saclay, and URA 2306, CNRS\\
91191 Gif-sur-Yvette Cedex, France
}
\email{ ayyer@math.ucdavis.edu}

\author[R. Cori]{Robert Cori}
\address{Robert Cori\\
Labri, Universit\'e Bordeaux 1 \\
33405 Talence Cedex\\
 France}
\email{cori@labri.u-bordeaux.fr}

\author[D. Gouyou-Beauchamps]{Dominique Gouyou-Beauchamps}
\address{Dominique Gouyou-Beauchamps\\
LRI, Universit\'e Paris 11 \\
91405 Orsay Cedex\\
 France}
\email{Dominique.Gouyou-Beauchamps@lri.fr}

\date{\today}

\begin{document}

\begin{abstract}
  We demonstrate a natural bijection between a subclass of alternating
  sign matrices (ASMs) defined by a condition on the corresponding
  monotone triangle which we call the gapless condition and a subclass
  of totally symmetric self-complementary plane partitions defined by
  a similar condition on the corresponding fundamental domains or
  Magog triangles. We prove that, when restricted to permutations,
  this class of ASMs reduces to 312-avoiding permutations. This leads
  us to generalize pattern avoidance on permutations to a family of
  words associated to ASMs, which we call Gog words.  We translate the
  gapless condition on monotone trangles into a pattern avoidance-like
  condition on Gog words associated.  We estimate the number of
  gapless monotone triangles using a bijection with p-branchings.
\end{abstract}

\maketitle

\begin{center}{\bf Dedicated to Doron Zeilberger on the \\
    occasion of his sixtieth birthday}\end{center}

\section{Introduction}\label{sec:introd}
The enumeration formula for ASM matrices was first proved by Doron
Zeilberger \cite{z1} by showing that monotone (or Gog) triangles naturally
associated to them are equinumerous to Magog triangles associated to
totally symmetric self-complementary plane partitions.  Since the
enumeration formula was known for these plane partitions this gave the
result. The proof needed a long list of computations and since then,
the construction of a bijection between Gog and Magog triangles has
been and still is an open problem.  We consider here a family of monotone
triangles, which we call gapless, for which we show a simple bijection
with a subclass of Magog triangles also satisfying a gapless
condition.

%We examine in detail the combinatorial structure of gapless Gog
%triangles and give an algorithm to build all of them.

%Since permutation matrices are ASMs, we consider the monotone
%triangles associated to them, it turns out that the gapless ones are
%exactly those correponding to permutation avoiding the pattern 312.

It turns out that among monotone triangles associated to permutations,
the gapless ones are exactly those correponding to permutation
avoiding the pattern 312.  This leads us to give a {\em pattern
  avoidance-like} characterization for ASMs whose monotone triangles are
gapless.  For this purpose we introduce a family of words associated
to monotone triangles and generalize to these words the notion of avoidance
of the pattern 312. We note in passing that this notion of pattern
avoidance for ASMs is different from that considered by Johansson and
Linusson \cite{jl}.

One of the main aims of this article is to build on the relationship
between alternating sign matrices and totally symmetric
self-complemen\-tary plane partitions first considered by Mills,
Robbins and Rumsey \cite{mrr1}. The other aim is to extend the notion
of pattern-avoidance to other combinatorial structures related to
permutations.  One of the first works in this direction known to us is
that of Hal Canary \cite{canary} who established a relation between
tilings of an Aztec diamond and alternating sign matrices using Baxter
permutations. 

The organisation of the paper is as follows. We first remind the
reader of monotone (or Gog) triangles in Section~\ref{sec:triang}. We
then define the class of gapless monotone triangles for which we
investigate their combinatorial structure and give an algorithm to
build them in Sections \ref{sec:shapes} and
\ref{sec:growingshapes}. We analyze a subset of these gapless
triangles of rectangular shape in Section~\ref{sec:rect}.  A natural
simple bijection from gapless monotone triangles to gapless Magog
triangles is then shown in Section~\ref{sec:magog}.  We next proceed
to show the relationship between monotone triangles and 312 patterns
for permutations in Section~\ref{sec:312perms}.  This will turn out to
be the simple case for the main theorem about 312 subpatterns in Gog
words, which we define in Section~\ref{sec:gogs}.  The proof of the
pattern-avoidance theorem is then given in
Section~\ref{sec:312gwords}. We summarize our results and give
directions for future work in Section~\ref{sec:conc}. We also give
asymptotic estimates of the number of gapless monotone triangles in
Appendix~\ref{sec:ngaplessasm} and, for completeness, leading
asymptotics of the number of alternating sign matrices in
Appendix~\ref{sec:nasm}.

\section{Monotone triangles} \label{sec:triang}

\begin{defn} \label{def:gog} 
  A {\em monotone triangle} or {\em strict Gelfand pattern} or {\em
    Gog triangle} of size $n$ is an array $\pi = (a_{ij})$ of positive
  integers defined for $n \geq i \geq j \geq 1$ that is written in the
  form
\be
\begin{array}{c c c c}
a_{1,1} &  &  &  \\
a_{2,1} & a_{2,2} &  &   \\
%&  & \cdots & \cdots\\
\vdots & \vdots & \ddots &  \\
a_{n,1} & a_{n,2} & \cdots & a_{n,n},
\end{array}
\ee
where,
\begin{enumerate}
\item $a_{i,j} < a_{i,j+1}$, $a_{i,j} \leq a_{i-1,j}$ and $a_{i,j} \leq a_{i+1,j+1}$ whenever
  both sides are defined. That is to say, numbers increasing along rows,
  decreasing weakly along columns, and decreasing weakly along diagonals. 
\item $a_{n,j}=j$.
\end{enumerate}
\end{defn}

We note that this is close to the original convention in \cite{mrr2}
for writing the triangles. However, it is different from the
convention in \cite{z1}.  Monotone triangles are in natural bijection
with alternating sign matrices \cite{mrr2}. Since permutation matrices
are a class of ASMs, one can express each permutation as a monotone
triangle. The way to do this is to write the $i$th row as an ascending
list of the first $i$ elements in the permutation. For example, the
permutation 4312 can be written as
\be \label{egtriangle}
\begin{array}{c c c c}
4 &  &  &  \\
3  & 4 &  &   \\
1 & 3 & 4 & \\
1 & 2 & 3 & 4,
\end{array}
\ee

We now define a special subset of monotone triangles; those with a gap.
\begin{defn} \label{def:forb} 
  A monotone triangle $a$ contains {\em a gap at
    position $(i,j)$} if $a_{i,j}-a_{i+1,j}>1$. A {\em gapless monotone
    triangle} is a monotone triangle which contains no gaps.
\end{defn}

In the above example \eqref{egtriangle}, the monotone triangle  
contains a gap at position $(2,1)$.  The monotone
triangle corresponding to the identity permutation is an extreme
example of a gapless one. The columnwise difference in that case is
zero everywhere.

In the definition of a monotone triangle, condition (1) consists in 
 the satisfaction of three  inequalities, notice that the third one can
 be omitted if the triangle is assumed to be gapless, since for integers 
the relation:
$$a_{i,j}< a_{i, j+1} \leq a_{i+1, j+1} + 1$$
 implies $a_{i,j} \leq a_{i+1, j+1}$.

\medskip

\noindent
More precisely we have:
\begin{rem}
A triangle $ (a_{ij})$ of positive
  integers, (where $n \geq i \geq j \geq 1$) is a gapless monotone triangle
if the following conditions are satisfied:
$
\begin{array}{ll}
(0) & \ \ \  a_{i,j} \leq a_{i-1,j} \leq a_{i,j} + 1 \\ 
 (1)'& \ \ \ a_{i,j} < a_{i,j+1}\\
 
(2) &\ \ \ a_{n,j}=j
\end{array}
$
\end{rem}

\section{Shapes of gapless monotone triangles} \label{sec:shapes}

To any gapless monotone triangle $(a_{i,j})$ of size $n$ we associate
a {\em shape}, that is an array $(s_{i,j})$ of size $n-1$ containing
entries 0 and 1 and defined by
$$ s_{i,j} = a_{i,j} - a_{i+1,j},$$
for $1 \leq j \leq i < n$.
For instance the shape of 
\be
\begin{array}{c c c c}
3 &  &  &  \\
2  & 4 &  &   \\
2 & 3 & 4 & \\
1 & 2 & 3 & 4,
\end{array}
\ee
is
\be \label{shapeg}
\begin{array}{c c c c}
1 &  &  &  \\
0  & 1 &  &   \\
1 & 1 & 1 &
\end{array}
\ee

We remark that any triangle $(s_{i,j})$ with 0,1 entries is not
necessarily the shape of a gapless monotone triangle. A simple
characterisation can be given by examining the partial sums of two
consecutive columns.

To any triangle $(s_{i,j})$ 
 of size $n-1$ with 0,1 entries we associate $n-1$  words 
$f^{(1)}, f^{(2)}, \cdots,  f^{(n-1)}$,
on the alphabet   $\{0,1\}$,
 representing its columns.

More  precisely the word  $f^{(j)}$ is 
of length $j$ and given by:
$$f^{(j)} = s_{n-1,n-j} \cdots s_{n-j,n-j}.$$

\begin{prop}\label{prop:consecutivecolumns}
Two words $g$ and $h$ on the alphabet  $\{0,1\}$ correspond to two consecutive
columns in the shape of a gapless monotone triangle if and only  their lengths
$|g|, |h|$  satisfy $|h| - |g| = 1$ and for any $i \leq |g|$ 
we have 
\be
 \sum_{k=1}^i g_k \geq \sum_{k=1}^i h_k
\ee
Moreover a sequence of words $f^{(1)}, f^{(2)}, \cdots f^{(n)}$ is
 the shape of a monotone
gapless triangle if and only if for any $k < n$ the words $f^{(k)}$ and $f^{(k+1)}$
satisfy the condition given above for $g$ and $h$.
\end{prop}

\begin{proof}
It suffices to remark that the $a_{i,j}$ of the
 corresponding monotone triangles are obtained
from the $f^{(i)}$ by the relation
$$ a_{i,j} \ \ = \ \ j + \sum_{k=1}^{n-i} f_k^{(n-j)}$$
so that the above inequality is a reformulation of the relation
$a_{i,j} < a_{i,j+1}$.
\end{proof}

\section{Growth processes for gapless shapes
%: column and row  insertion
}\label{sec:growingshapes}

In the sequel {\em gapless shape} will mean shape of a gapless monotone
triangle.  That is a triangle $s_{ij}$ with entries 0 and 1 which
columns satisfy the condition in Proposition
\ref{prop:consecutivecolumns}.

\subsection{Column insertion}
A natural way to build a gapless shape  is to proceed in giving the words
$f^{(1)}, f^{(2)}, \cdots, f^{(n)}$ successively. The choice of possible
$f^{(n)}$ depends only on the value of $f^{(n-1)}$. Moreover all the 
 $f^{(n)}$ which may follow a given 
$ f^{(n-1)}$  can be build  by using a simple process which we describe
below.

 For that purpose we consider some  words on the alphabet  $\{a, b\}$.
We recall that a word $w$ on this alphabet is a Dyck word
if for any prefix $w'$ of $w$  the number
 of occurences of the letter $a$ (denoted $|w'|_a$) 
is not less than the number of occurrences
of the letter $b$ (denoted $|w'|_b)$.
We now 
introduce a mapping $\phi$ associating 
to two words 
of length $n-1$ and $n$ on the alphabet $\{0,1\}$ a word of length 
$2n-1$ on the alphabet  $\{a, b\}$.

We give first  a mapping $\theta$ giving for each pair of letters
in $\{0,1\}$ a word of length 2 on the alphabet $\{a,b\}$: 

$$ \theta(0,0) = ab, \ \  \theta(1,1) = ba, \ \   \theta(1,0) = aa, \ \ 
 \theta(0,1) = bb.$$

Considering  two words $g$ and $h$ on the alphabet $\{0,1\}$ of lengths 
$|h| = |g|+1 = n $ 
 we associate a
word $w=\phi(g,h)$ of length $2n-1$ on the alphabet $\{a,b\}$ given by
$$w=a \theta(g_1,h_1) \theta(g_2,h_2) \cdots \theta(g_{n-1},
h_{n-1}).$$
For example, the word corresponding to the first two columns in the
shape \eqref{shapeg}, $g=(1,1)$ and $h=(1,0,1)$, is
$$w=\phi(g,h)=a \theta(1,1) \theta(1,0) = abaaa,$$
which is a valid Dyck word.

\begin{prop}\label{prop:nboccurences}
Two words  $g$ and $h$ of lengths $n-1$ and $n$ respectively 
correspond to two consecutive columns in a gapless  shape 
if and only if the word $ w = \phi(g,h)$, 
 is a prefix of a Dyck word.

\end{prop}

\begin{proof}
The condition on the partials sums of $g$ and $h$ given in Proposition
\ref{prop:consecutivecolumns} translates naturally into counting the
number of occurences of $a$ and $b$ in $\phi(g,h)$.
\end{proof}

\begin{cor}
The number of pairs of words $g,h$ such that $$|h| = |g|+1 = n$$ and 
$g$ and $h$ represent consecutive columns in a gapless shape 
is equal to the central binomial coefficient
$$\binom{2n}{n}$$
\end{cor}
\begin{proof}
Using the well-known fact that the number of prefixes of Dyck words of
length $2n-1$ is equal to the binomial coefficient $\binom{2n-1}{n}$,
and that the last letter of $h$ may be 0 or 1, we get
$2\binom{2n-1}{n} = \binom{2n}{n}$.
\end{proof}

We used the column insertion construction to compute the number of all
gapless shapes of monotone triangles up to size 12 obtaining the
numbers below. For comparison, we also list the total number of ASMs.

\begin{table}[h!] 
\begin{tabular}{|c|c|c|}
\hline
Size & Number of ASMs in bijection & Total number of ASMs \\
\hline
1 & 1 & 1\\
2 & 2 & 2\\
3 & 6 & 7\\
4 & 26 & 42\\
5 & 162 & 429 \\
6 & 1450 & 7436 \\
7 & 18626 & 218348 \\
8 & 343210 &  10850216  \\
9 & 9069306 & 911835460 \\
10 & 343611106 &  129534272700\\ 
11 & 18662952122 & 31095744852375\\
12 & 1453016097506 & 12611311859677500 \\
\hline
\end{tabular}
\vspace{0.3cm}
\caption{The number of ASMs in bijection and the
  total number of ASMs.} \label{tab:bijasm}
\end{table}
\vspace{-1cm}

\section{Rectangular shapes} \label{sec:rect}
Although it seems difficult to obtain a closed formula for the number
of triangular shapes, the number of rectangular ones have a nice
enumeration formula.

\begin{defn}
A gapless rectangular shape is a rectangular  array $r_{i,j}$ where
$1 \leq i \leq p$, $ 1 \leq j \leq m$
 of 0's and 1's such that for any $j, i$ (where $1 \leq  j < m$ and 
$1 \leq i \leq p$),
\be
 \sum_{k=i}^p r_{k,j} \leq \sum_{k=i}^p r_{k,j+1}.
\ee
\end{defn}

To any  rectangular shape  can be associated a cumulant
rectangular array $s_{i,j}$ such that 
$s_{i,j} = \sum_{k=i}^p r_{k,j}$.
Then the condition on $r$ translates on $s$ into
$$ s_{i,j} \leq s_{i,j+1}.$$
Moreover since $r_{i,j} \in \{0,1\}$,
$$s_{i,j+1} -1 \leq s_{i,j}  \leq p-i+1.$$
An example of a gapless rectangle and its cumulant array is given in
Figure~\ref{fig:gaplessrec}. 
\begin{figure}[h]
\[
\begin{array}{cccccccccccccccccccc}
0 & 0 & 0 & 0 & 1 & 1 & 1 & 1 & 0  & \ \ \ \ \ \ \ \ \ \ \ & 0 & 1 & 1 & 1&
3 & 3 & 3 & 3 & 3 \\
0 & 0 & 0 & 0 & 1 & 1 & 1 & 0 & 1 & &  0 & 1 & 1 & 1 & 2 & 2 & 2 & 2 & 3\\
0 & 1 & 1 & 1 & 1 & 0 & 0 & 1 & 1 &  & 0 & 1 & 1 & 1 & 1 & 1 & 1 & 2 & 2 \\
0 & 0 & 0 & 0 & 0 & 1 & 1 & 1 & 1 & & 0 & 0 & 0 & 0 & 0 & 1 & 1 & 1 & 1\\
\end{array}
\]
\caption{A gapless rectangle and its cumulant array}  \label{fig:gaplessrec}
\end{figure}

To any cumulant array is associated a triangular array $t_{i,j}$ such
that $t_{i,j}$ is the smallest $k$ such that $s_{n+2-i-j,k} \geq j$ if
such $k$ exists and equal to $m+1$ otherwise.  The triangular array
associated to the cumulant array in Figure~\ref{fig:gaplessrec} is
given by
\be \label{brancheg}
\begin{array}{cccc}
6 & 8 & 9 & 10\\
2 & 5 & 5 & \\
2 & 5 & & \\
2 & & &  
\end{array} 
\ee
It turns out that the  triangular array built above belongs to a
family  of arrays called branchings and  considered
in \cite{gelbart72} and in  \cite{carstan}.

\begin{defn}\label{def:branchings}.
A $p$-branching is a triangular array of postive integers
$b_{i,j}$ where $1\leq i\leq p$,  $1 \leq j \leq p+1-i$
such that the entries are weakly increasing along rows and weakly
decreasing along columns,
$$b_{i+1,j} \leq b_{i,j} \leq b_{i,j+1}$$
\end{defn}

The relation between gapless rectangular shapes, cumulant arrays and
$p$-branchings may be summarized in the following statement.
\begin{prop}
The above construction is a bijection between gapless 
rectangles of size $p, m$ and $p$-branchings with entries
less or equal to $m+1$.
\end{prop}

Recall that a semi-standard Young tableau is an array $t_{i,j}$ of entries
 weakly increasing along rows, strictly increasing along columns, that is
$$t_{i,j} \leq t_{i,j+1} \text{ and } t_{i,j} < t_{i+1,j}.$$
 Then the following proposition allows us to obtain an enumeration 
formula for gapless rectangles.

\begin{prop}
Let $p,m$ be two integers not less than 1.  There is a bijection
between $p$-branchings with entries not greater than $m$ and the set
of semi-standard Young tableaux with entries in $1,2,\ldots p$ and
having less than $m$ columns.
\end{prop}

\begin{proof}
Consider the branching $B=(b_{i,j})$, we build a semi-standard
Young tableau $T$ associated to it performing the following algorithm:

For each $i$ the $i$-the row of $T$ has length $k_i-1$ where $k_i$ is
the largest value appearing in the $i$th row of $B$. Then $t_{i,j}$ is
equal to $p+1$ minus the number of elements in the $i$-th row of $B$
which are not less than $j+1$.  The fact that $T$ is indeed a
semi-standard Young tableau follows from two simple remarks. The first
one shows that $t_{i,j} \leq t_{i,j+1}$ and the second shows that
$t_{i,j} < t_{i+1,j}$.
\begin{itemize}
\item Let $u= u_1, u_2, \ldots, u_i$ be any sequence of integers and
  let $a_k$ be the number of elements in $u$ not less than a given
  integer $k$ then $a_k \geq a_{k+1}$.
\item Let $u= u_1, u_2, \ldots u_i$ and $v = v_1, v_2, \ldots,
  v_{i-1}$ be two non decreasing sequences of integers such that for
  all $j < i$, $u_j\geq v_j$ and let $a_k$ (resp. $b_k$) be the number
  of elements in $u$ (resp. $v$) not less than a given integer
  $k$. Then $a_k > b_k$ for all $k$.
\end{itemize}

\medskip

Conversely given $p,m$ and a semi-standard Young tableau $T$ with
entries not greater than $p$ and having at most $m$ columns we build
the branching $B$ by the process described below.

For each $i$ from $1$ to $p$, the $i$-th row of $B$ (which has length
$p+1-i$) consists of $t_{i,k}-t_{i,k-1}$ entries equal to $k$ for each
$k >1$. If the length of the $i$th row of $T$ is $l_i$, fill in
$p+1-t_{i,l_i}$ entries equal to $l_i+1$. Since the $i$th row of $B$
has to have length $p+1-i$, we fill the row with remaining entries
equal to 1 and arrange the entries in non-decreasing order.

\end{proof}

To illustrate the above bijection we give the semi-standard Young
tableau associated to the branching in \eqref{brancheg}, 
\[
\begin{array}{ccccccccc}
1& 1& 1& 1& 1& 2 & 2 & 3 & 4\\
2 & 3 & 3 & 3 & \\
3 & 4 & 4 & 4 & \\
4. & & &  
\end{array} 
\]

\begin{cor}
The number $\rho_{m,p}$ of rectangular gapless shapes with $p$ rows
and $m$ columns is equal to,
$$\rho_{m,p}\ \ = \ \ \prod_{i=1}^p\frac{ \binom{m+2i-1}{i}}
{\binom{2i-1}{i}} 
= \prod_{i=1}^p \prod_{j=i}^p\frac{m+i+j-1}{i+j-1}.$$
\end{cor}
\begin{proof}
The number of such rectangles is by the above bijections equal to the
number of semi-standard Young tableaux with entries in $1,2,\ldots, p$
and having less than $m+1$ columns, which is also the number of
column-strict partitions (see \cite{stanley71}).
\end{proof}

\subsection{Decomposition of a triangular gapless shape}

For any integer $k \leq n$ triangular gapless shape of size $n$ can be
decomposed into a rectangle with $k$ rows (and $n-k$ columns) and two
triangles of sizes $k$ and $n-k$ as shown in
Figure~\ref{fig:tridecomp}.
\begin{center}
\begin{figure}[H]
\includegraphics{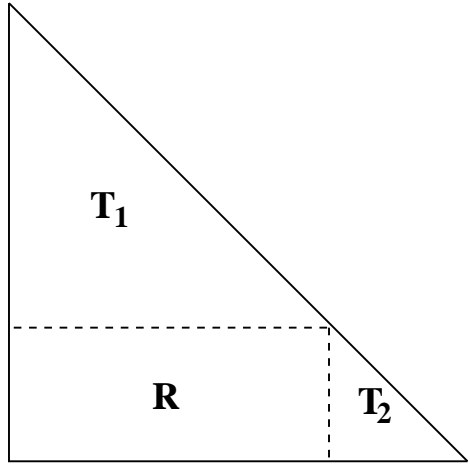} 
\caption{Decomposition of a gapless triangle} \label{fig:tridecomp}
\end{figure}
\end{center}
From this decompostion we get the following asymptotic estimate on the
number of gapless triangular shapes.
\begin{prop}\label{cor:lowerbound}
The number $a_n$ of  gapless triangular shapes of size
$n$ satisfies the inequality:
$$ a_{n} \geq \rho_{n-k,k} a_{n-k}$$
where  $\rho_{m,k}$ denotes the number of rectangular gapless shapes
with $k$ rows and $m$ columns.
\end{prop}
\begin{proof}
Given a rectangular gapless shape {\bf R}  with $k$ rows 
and $n-k$ columns, we build a triangular  gapless shape {\bf $T$} adding above
any triangular shape {\bf $T_1$} and on its right a triangular shape {\bf $T_2$} in which all entries are equal to 1.
\end{proof}

%Notice that the last row of a shape $(s_{i,j})$ of size $n$ of a
%gapless triangle (of size $n+1$) consists of a sequence of $p$
%entries equal to 0 followed by $n-p$ entries equal to 1 (where $0
%\leq p \leq n$). The row above it consists in a sequence of $q_1$
%entries equal to 0 followed by $p-q_1$ entries equal to 1 then $q_2$
%entries equal to 0 followed by $n-p-q_2-1$ entries equal to 1.

A lower bound for an asymptotic evaluation purpose may be obtained
when taking $k = \frac{n}{2}$.
\begin{cor}\label{cor:lowerbound2}
Consider the sequence of integers given by
$$ \alpha_1 = 1, \alpha_2 = 2, \alpha_3 = 6, \alpha_4 = 26$$ and
$$ \alpha_n = \rho_{\lceil \frac n2 \rceil, \lfloor \frac n2 \rfloor}
\alpha_{\lceil \frac n2 \rceil}.$$  
Then the number of gapless shapes of size $n$ is greater than $\alpha_n$.
\end{cor}

However, numerics  suggest that the true estimate is
\be \label{angrowth}
a_n \sim  \left( \frac{3^9}{2^{12}} \right)^{\frac{n^2}{8}}.
\ee
It would be interesting to derive this from a formula for $a_n$.

\section{Magog Triangles}\label{sec:magog}
We now prove a bijection between Magog triangles (repesenting totally
symmetric self-complementary plane partitions) which avoid a certain
pattern and 312 avoiding Gog words.

\begin{defn} \label{def:magog}
A {\em Magog triangle or fundamental domain for partitions}  of size $n$
is an array $(b_{ij})$
of positive integers defined for $n \geq i \geq j \geq 1$ that is written in
the form
\be
\begin{array}{c c c c}
b_{1,1} & \\
b_{2,1} & b_{2,2} & \\
\vdots & \vdots & \ddots &  \\
b_{n,1} & b_{n,2} & \cdots & b_{n,n},
\end{array}
\ee
where,
\begin{enumerate}
\item $b_{i,j} \leq n$,
\item $b_{i,j} \leq b_{i+1,j}$ and $b_{i,j} \leq b_{i,j+1}$ whenever
  both sides are defined (weak decreasing along both directions),
\item $b_{i,j} \geq i$.
\end{enumerate}
\end{defn}

\begin{defn} \label{def:spmagog}
  A Magog triangle $b$ contains {\em a gap at position
  $(i,j)$} if $b_{i+1,j}-b_{i,j}>1$.  A {\em gapless Magog triangle}
  is a Magog triangle which contains no gaps.
\end{defn}

Notice that the above condition is very similar to that defining
gapless monotone triangles in Definition~\ref{def:forb}. We showed in
Section~\ref{sec:312gwords} that Gog words which avoid the generalized
312 pattern are in bijection with gapless monotone triangles, but we have
not been able to find an interpretation for gapless Magog
triangles. Since Magog triangles were motivated by and are in
bijection with totally symmetric self-complementary plane partitions,
one could hope for a natural interpretation in that context.

Consider the following transformation $\Delta$ associating to any 
monotone triangle $(a_{ij})$ with integer coefficients a new triangle
 $(b_{ij})$ given by $$b_{i,j}\ \   = \ \ a_{i,j}+i-j$$

\begin{thm}
  For a monotone triangle $T = (a_{ij})$ the triangle $\Delta(T)$ is a
  Mogog triangle if and only if $T$ is gapless. Furthermore,
  $\Delta(T)$ is a gapless Magog triangle.
\end{thm}
\begin{proof}
  The proof is straightforward and left to the reader.
\end{proof}

We give an example to illustrate the idea. The gapless monotone triangle,
\be
\begin{array}{c c c c}
2 &  &  &  \\
2 & 3 &  &  \\
1 & 2 & 4 &   \\
1 & 2 & 3 & 4,
\end{array}
\ee
maps to the gapless Magog triangle,
\be
\begin{array}{c c c c c c c c}
2 &  &  &  \\
3 & 3 &  &  \\
3 & 3 & 4 &  \\
4 & 4 & 4 & 4.
\end{array}
\ee

Notice that whenever the difference between the entries at positions
$(i,j)$ and $(i+1,j)$ in the gapless monotone triangle was $\delta$
($\delta =0,1$ by definition), the difference between the entries at positions
$(i+1,j)$ and $(i,j)$ in the Magog triangle was $1-\delta$. This is an
obvious consequence of the bijection.

\section{312 patterns for permutations} \label{sec:312perms}

We will now show a permutation avoids the pattern 312 if and only if
the some monotone triangle associated to it is gapless. First, we
remind the reader of the definition of pattern avoidance for the
specific case of the pattern 312.

\begin{defn}
A permutation $\pi$ of the letters $\{1,\dots,n\}$ {\em contains the
pattern 312} if there exists positive integers $i,j,k$, with $1\leq
i<j<k \leq n$ such that $\pi_i>\pi_k> \pi_{j}$. If no such triple of
integers exist, the permutation $\pi$ is said to {\em avoid the pattern 312}.
\end{defn}

%One can show that the number of permutations of size $n$ that avoid
%the pattern 312 (and moreover any other permutation of size 3) is
%equal to the Catalan number.  

Before we demonstrate the bijection, we need a simple but useful observation.

\begin{rem} \label{rem:312}
If a permutation $\pi$ contains the pattern 312 at positions
$(i,j,k)$ and $j-i>1$, then there exists another integer $l$, $i \leq
l \leq j$ such that the permutation contains another pattern 312 at
$(l,l+1,k)$. The proof of this comes by looking at the integer
$\pi_{i+1}$. If $\pi_{i+1}<\pi_k$ then $(i,i+1,k)$ is the required
triple. If not, then $(i+1,j,k)$ is the new triple and we iterate.
\end{rem}

For example, the permutation $\pi={\underline 5}42 {\underline 1}6
{\underline 3}$  has a 312 pattern as shown. The iteration is as
follows,
\be
{\underline 5}42 {\underline 1}6 {\underline 3} \to 
5{\underline 4}2 {\underline 1}6{\underline 3} \to 
5{\underline 4} {\underline 2}16{\underline 3}.
\ee
We use the above observation  to refine the definition of containment of
the 312 pattern.
\begin{defn}
  We say that a permutation $\pi$ {\em contains the 312 pattern at
    position $i$} if $\pi$ contains a 312 pattern at positions
  $(i,i+1,k)$ for some $k>i+1$. 
\end{defn}

Remark~\ref{rem:312} leads to the following proposition.

\begin{lem} \label{lem:permconsec}
Assuming that a permutation $\pi=(\pi_1,\dots,\pi_n)$ does not contain
a 312 pattern at any positions less than $i$, it avoids the 312 pattern at
position $i$  if and only if the set $S_i=\{\pi_j| j \leq i+1, \pi_j \geq
\pi_{i+1} \}$ consists of consecutive integers.
\end{lem}

\begin{proof}
  % The main idea is that a permutation $\pi$ contains the 312 pattern
  % at position $i$ if the difference between $\pi_i$ and $\pi_{i+1}$
  % is greater than one and moreover, there exists a $k$ such that
  % whenever $\pi_{i+1} < \pi_k < \pi_i$, then $k>i+1$. In plain
  % english, whenever $\pi_i$ is the ``3'' and $\pi_{i+1}$ is the
  % ``1'' of the 312 pattern, there must be a ``2'' which occurs after
  % the ``1''. 

Suppose there are no 312 patterns at positions less than $i$. The only
way to get a 312 pattern at position $i$ is to have
$\pi_{i}>\pi_{i+1}+1$ and have a $k$ such that if $\pi_{i+1} < \pi_k <
\pi_i$, then $k>i+1$. But if $S_i$ is consecutive, all such $k$'s will
necessarily satisfy $k<i$.

For the converse, assume that $S_i$ is not consecutive. Therefore,
there exists a smallest $k$ such that $\pi_k>\pi_{i+1}$ and $\pi_k
\notin S_i$. Now, if $\pi_i < \pi_k$, then $S_{i-1}$ is not
consecutive either, which is a contradiction. 
%\marginnote{Is this correct?}

Therefore, we must have $\pi_i > \pi_k$, in which case
$(\pi_i,\pi_{i+1},\pi_k)$ is the required 312 pattern.
\end{proof}

As a consequence of Lemma~\ref{lem:permconsec}, we need to only look
at the prefix of a permutation $\{\pi_j| j \leq i+1, \pi_j \geq
\pi_{i+1} \}$ to see if it contains a 312 pattern. Of course, to
ensure that the permutation avoids 312, we must look at the entire
permutation.

\begin{lem} \label{lem:perm312}
A permutation $\pi$ on $n$ letters contains the pattern 312 at
position $i$ (assuming no 312 pattern at positions less than $i$) if
and only if there exists a gap in the corresponding monotone triangle
$a$ at row $i$. 
\end{lem}

\begin{proof}
Consider rows $i,i+1$ for the monotone triangle $a$. Suppose
$a_{i,j}=\pi_{i+1}$  and look at columns starting $j$ because the
columns from 1 to $j-1$ will be unchanged. Since
$\pi_{i+1}$ is the only new entry in row $i+1$, all the entries in row
$i$ larger than it will be shifted to the right by one in column $i+1$
as shown,
\be
\begin{array}{c c c c c }
a_{i,j} & a_{i,j+1} & \dots &  a_{i,i} &\\
\hspace{1cm}\searrow  &  & \dots & \hspace{1cm}\searrow &  \\ 
\underbrace{a_{i+1,j}}_{\pi_{i+1}} & a_{i+1,j+1}
 & \dots & \dots & a_{i+1,i+1}.
\end{array}
\ee
Notice that the set $\{a_{i+1,j},\dots,a_{i+1,i+1}\}$ equals $S_i$. 
As a consequence of
Lemma~\ref{lem:permconsec}, it is clear that there will be no gaps if
and only if this set is consecutive.

\end{proof}

\begin{rem} \label{rem:double312}
If a permutation contains the pattern 312 at position $i$, then there
could be multiple columns where the gap occurs at row $i$. For
example, the permutation $3\underline{514}2$ contains
a 312 at the second position but the monotone triangle
\be 
\begin{array}{c c c c c}
3 &  &  &  &\\
\underline 3  & \underline 5 &  & &  \\
1 & 3 & 5 & &\\
1 & 3 & 4 & 5 &\\
1 & 2 & 3 & 4 &5,
\end{array}
\ee
contains two gaps in the second row.
\end{rem}

\section{Gog words} \label{sec:gogs}
We now define words which are enumerated by the ASM numbers and are in
natural bijection with alternating sign matrices. Before that, we need
an alphabet, the set of which itself depends on $n$.

\begin{defn} \label{def:alph}
The {\em Gog alphabet}, $\mathcal A$, is given by tuples 
$(p_1,q_1,\dots,$ $p_{k-1},q_{k-1},p_k)$ such that each $1 \leq p_1 < q_1
\dots p_{k-1} < q_{k-1} <p_k$.
\end{defn}

For example, $\mathcal A$ contains all the integers but in addition
contains elements such as $(1,2,3)$ and $(7,9,10,13,17)$. In all
examples, we will avoid commas between the entries in a tuple
throughout the article.

\begin{defn} \label{def:word}
A {\em Gog word} $w$ of size $n$ is a word consisting of elements of
$\mathcal A$ satisfying the following conditions:
\begin{enumerate}
\item The length of the word is $n$,
\item Each letter in the word has maximum entry at most $n$,
\item An integer in an even numbered position in a tuple is repeated
  in another tuple to its left and to its right in odd numbered
  positions,
\item Every repeated integer alternates in odd and even numbered
  positions in subsequent tuples. 
\end{enumerate}
\end{defn}

Notice that the first and last letter has to be of size one.
Clearly, Gog words with alphabets of size one correspond to
permutations.  We omit the brackets around singleton elements (i.e.,
we write $2$ instead of $(2)$) for ease of reading.  For example
$2(123)2$ and $213$ are both Gog words of size $3$ and
$25(12456)(345)(234)3$ is a Gog word of size $n=6$. 

A word about notation. We will denote elements of the alphabet
$\mathcal A$, 
also called tuples,
 by $x_i$, the lengths of these elements by $2k_i-1$, odd elements of
 these tuples by $p_j$ and even elements by $q_j$. Elements of the monotone
triangle will be denoted using the same subscripts, e.g. $a_{i,j}$.

\begin{rem} \label{rem:gogasms}
  There is a natural bijection relating Gog words and alternating sign
  matrices. The $j$th tuple of a Gog word is simply the positions, in
  increasing order, of nonzero entries in the $j$th row in the
  corresponding alternating sign matrix. The even numbered positions
  correpond to -1 and the odd numbered ones, to +1.
\end{rem}

\begin{rem} \label{rem:triangword}
There is a natural bijection relating Gog words and monotone
triangles. Given a Gog word $x_1 \dots x_n$, the first row of the
corresponding monotone triangle is simply the first element and the
$i$th row is obtained by taking the $i-1$th row, removing the
even-positioned elements in $x_i$ from it, appending the
odd-positioned elements in $x_i$ to it, and sorting in increasing
order.

Given a monotone triangle $a$, the first tuple is the singleton
element in the first row, and the $i$th tuple in the
corresponding Gog word is given by listing the sequence
\be
(a_{i,1}, a_{i-1,1}, \dots, a_{i,i-1}, a_{i-1,i-1}, a_{i,i})
\ee
and removing occurences of every repeated integer. 
\end{rem}

For an example, consider the monotone triangle,
\be \label{egasm}
\begin{array}{c c c c}
3& &  &  \\
1 & 3 &  &\\
1 & 2 & 4 &  \\
1 & 2 & 3 & 4.
\end{array}
\ee
We construct first the lists of sequences $(3)(133)(11234)(1122344)$
and then remove every occurence of repeated integers to get $(3)(1)(234)(3)$.
Both the bijections in Remark~\ref{rem:gogasms} and
Remark~\ref{rem:triangword} are commensurate with the well-known
bijection relating monotone triangles and ASMs \cite{mrr2}.

Notice that a Gog word is not a permutation of letters in the alphabet
$\mathcal A$, as the example $2(123)2$ shows.

%We will however show subsequently  that there is a natural analog of
%a 312 avoiding Gog word. 

\section{312 subpatterns for Gog words} \label{sec:312gwords}
We showed in Lemma~\ref{lem:perm312} that a 312 pattern in a
permutation is equivalent to a column-wise difference greater than one
in the corresponding monotone triangle.  A natural generalization
would be to interpret the
occurence of this columnwise difference in an arbitrary monotone
triangle as the occurence of a generalized  312 pattern in the
corresponding Gog words.  

For example, the monotone triangle in \eqref{egasm}
corresponds to the word $31(234)3$ and the fact that the $(1,1)$th
entry minus the $(1,2)$th entry is greater than one can be interpreted
as an occurence of a generalized 312 pattern. We now make this
notion precise.

\begin{defn}
We say that an integer $m$ is {\em active} with respect to a tuple
$x=(p_1,q_1,\dots,$ $p_{k-1},q_{k-1},p_k)$ if $m>p_k$ or if
$p_{j}<m<q_{j}$ for $j \in [k-1]$.
\end{defn}

\begin{defn} \label{def:gword}
  We say that integers $c,a,b$ form a {\em 312-subpattern} of the Gog word
  $w=x_1 x_2 \dots x_n$ if the following conditions hold
  \begin{enumerate}
  \item $c,a,b$ appear in odd positions in $x_i,x_j,x_k$ respectively
    where $i<j<k$,
  \item $b$ is not in an even position in $x_{i+1},\dots,x_{k-1}$,
  \item $b$ is active with respect to $x_j$.
  \item $a<b<c$.
  \end{enumerate}
\end{defn}

In the above example \eqref{egasm}, ${\underline 3}{\underline 1}({\underline
2}34)3$ contains the 312-subpattern.  The following two examples
illustrate the subtlety of the activity condition.  $$2 {\underline
5}({\underline 1}2356)5{\underline 4}2$$ contains a 312-subpattern but
$$2{\underline 5}({\underline 1}2456)5{\underline 3}2$$ does not.

This definition is stronger than the one used in \cite{jl}, whose
conditions, in this language, amount to $(1)$ and $(4)$ in
Definition~\ref{def:gword}. Therefore, the number of Gog words
avoiding the 312-subpattern is bounded {\em below} by the number of
alternating sign matrices avoiding the permutation matrix of the
permutation 312.

\begin{rem} \label{rem:312gog}
  If a Gog word $w=x_1 \dots x_n$ contains the 312-subpattern
  $(c,a,b)$ at positions $x_i,x_j,x_k$ whose lengths are
  $2p_i-1,2p_j-1$ and $2p_k-1$ respectively, then
  $((x_i)_{2p_i-1},(x_j)_1,b)$ is also a 312-subpattern since
  $(x_i)_{2p_i-1} \geq c$ and $(x_j)_1 \leq a$. Notice that the
  activity condition is unaffected.
\end{rem}

A natural question would be whether Remark~\ref{rem:312} generalizes
to Gog words. The following lemma shows that it does.

\begin{lem}
If a Gog word $w=x_1 \dots x_n$ contains the 312-subpattern, then there exist
positions $i,k$ such that the subpattern $(c,a,b)$ appears in
$x_i,x_{i+1},x_k$.
\end{lem}

\begin{proof}
The idea of the proof is very similar to that in
Remark~\ref{rem:312}. Suppose that the 312-subpattern in $w$ occurs at
$x_i,x_j,x_k$ respectively. It follows from Remark~\ref{rem:312gog}
that we can take $c=(x_i)_{2p_i-1}$ and $a=(x_j)_1$. Now look at the
last entry in $x_{i+1}$, $p=(x_{i+1})_{2p_{i+1}-1}$. 

First, note that $p$ cannot be equal to $b$ because $b$ is present in
an odd position in $x_k$ and therefore, that would force $b$ to be in
an even position between $x_{i+1}$ and $x_k$, which is forbidden in
the definition of the 312-subpattern.

If $p>b$, then $(p,a,b)$ is a 312-subpattern and we iterate. If $p<b$,
then $(c,p,b)$ is a possible 312-subpattern if we can show that $b$ is
active with respect to $x_{i+1}$. But this is guaranteed since $b$ is
larger than $p$, which is the largest element in $x_{i+1}$. This
proves the lemma.
\end{proof}

We will now generalize Lemma~\ref{lem:permconsec} from permutations to Gog
words.

\begin{lem} \label{lem:gogconsec}
Let $w=x_1 \dots x_n$ be a Gog word and $a$ be the corresponding
monotone triangle. Let $x_{i+1} = (p_1,q_1,\dots,p_{k-1},q_{k-1},p_k)$.
Define
\be
\begin{split}
&P_l = \{ a_{i,j} | q_l \geq a_{i,j} > p_l \}, \quad \text{for $l<k$ and} \\
&P_k = \{ a_{i,j} | a_{i,j} \geq p_k \}.
\end{split}
\ee
Then there is no gap at row $i$ in $a$ if and only if $P_l \cup
\{p_l\}$ for $l< [k]$ and $P_k$ consist of consecutive integers.
\end{lem}

\begin{proof}
We begin with an innocuous observation.  Let $w=x_1 \dots x_n$ be a
Gog word of size $n$ and $a$ be the corresponding monotone triangle.
Let $x_j = (p_1,q_1,\dots,p_{k-1},q_{k-1},p_k)$. For any $i$, if
$p_i = a_{i,r}, q_i = a_{i-1,s}$ and $p_{i+1} = a_{i,t}$, then $r
\leq s < t$.

This is simply a consequence of the definition of a monotone
triangle. Suppose $s<r$, and part of the monotone triangle looks
like
\be
\begin{array}{c c c c}
\underbrace{a_{i-1,s}}_{q_i} & \dots & a_{i-1,r-1} & a_{i-1,r} \\
% &  &  & \searrow & \dots & \searrow &\\
a_{i,s} & \dots & a_{i,r-1} & \underbrace{a_{i,r}}_{p_i} 
\end{array}
\ee
Then $q_i \leq a_{i-1,r-1} \leq p_i$, which is a
contradiction. Similarly, suppose $t \leq s$, from which part of the
monotone triangle looks like
\be
\begin{array}{c c c}
a_{i-1,t} &\dots &  \underbrace{a_{i-1,s}}_{q_i}\\
% &  &  & \searrow & \dots & \searrow &\\
 \underbrace{a_{i,t}}_{p_{i+1}} & \dots & a_{i,s}
\end{array}
\ee
which implied $p_{i+1} \leq a_{i-1,t} \leq q_i$, which is again a
contradiction.

From this it follows that the set of elements $\{a_{i-1,j}|r \leq j <s
\}$ is exactly the same as the set $P_i$,
\be
\begin{array}{c c c c c}
a_{i-1,r} & a_{i-1,r+1} &\dots &  a_{i-1,s-1} &q_i\\
  \hspace{1cm}\searrow & \dots & \dots &\hspace{1cm}\searrow \\
p_{i} & a_{i,r+1} &\dots & a_{i,s-1} & a_{i,s}
\end{array}
\ee
because none of the elements in $P_i$ is
removed. It is now clear that if the elements in $P_i$ are
consecutive, then there cannot be a gap between columns $r$ and $s$. 
This argument works for all $i<k$. For $i=k$, the argument is
identical to that in the proof of Lemma~\ref{lem:perm312}.
\end{proof}

We are now in a position to generalize Lemma~\ref{lem:perm312} to
Gog words.

\begin{thm} \label{thm:312}
Assuming that a Gog word $w$ of size $n$  does not contain the
312-subpattern at positions less than $i$, it
contains the 312-subpattern at position
$i$ if and only if  there exists a gap in the
corresponding monotone triangle $a$ at row $i$.
\end{thm}

\begin{proof} 
Combining what we have already shown, it remains to show, using the
notation of Lemma~\ref{lem:gogconsec} that the
existence of a 312 pattern of row $i$ is equivalent to the existence
of a $P_l$ which is not consecutive.

Suppose the 312-subpattern is obtained by the triple of integers
$(c,a,b)$. We first show that $b$ cannot be any of the $p_l$'s or
$q_l$'s. Clearly $b$ cannot be equal to any $q_l$ because then it is
in an even position in $x_{i+1}$, which is forbidden by
definition. $b$ cannot also be equal to any $p_l$ because it will be
forced to be in an even position strictly between $x_{i+1}$ and $x_k$
as a consequence of being a Gog word, which is forbidden for the same reason.

If $p_k > b$, $P_k$ is not consecutive and we are done. Otherwise, by
the activity condition, there exists $l$ such that $p_l \leq b \leq
q_l$ and then $P_l$ is not consecutive. As a consequence of
Lemma~\ref{lem:gogconsec}, we have a gap at row $i$.

For the converse, suppose there is a gap at row $i$. By
Lemma~\ref{lem:gogconsec}, either $P_l$ for $l<k$ or $P_k$ is not
consecutive. In either case, there exists an
integer $x$ in row $i$ and an integer $z$ in row $i+1$ in the same
column as $x$ such that $x-z>1$. Choose an integer $b$ such that $x<b<z$. By
choice, $b$ is active with respect to $x_{i+1}$. Let $x_m$ be the
first tuple in which $b$ appears in an odd position after
$x_{i+1}$. By construction, $b$ is not in an even position in tuples
between $x_{i+1}$ and $x_m$. We set $a=p_1$ and $c$ to be the largest
integer in $x_i$. Clearly $a<b$ as $a<x$. 

All that remains to be done is to prove $b<c$. We will show that by
demonstrating that the column in which $c$ belongs cannot be to the
left of that of $x$ and $z$. If that happens, rows $i-1,i$ and $i+1$
will look as follows,
\be
\begin{array}{c c c c c}
* & \dots &  x & * & \dots\\
%   &  &\hspace{1cm}\searrow & \hspace{1cm}\searrow & \dots\\
c  &\dots &  z  & x & \dots\\
%   & \dots &\hspace{1cm}\searrow & \dots\\
* & \dots & * & z & \dots,
\end{array}
\ee
and as shown, there will be a gap between rows $i-1$ and $i$ by the
same $x$ and $z$, which is forbidden by assumption.
Therefore $c$ cannot be in a column to the left of that of $x$ and
$z$. It can, of course, be in the same column (i.e. $x=c$), or to the right. In
both cases $b<c$ and thus, $(c,a,b)$ is the required 312 pattern.
\end{proof}

\section{Conclusions} \label{sec:conc}
We have defined a new set of words on an infinite alphabet which are
equinumerous with alternating sign matrices. Further, we showed that
there is a notion of 312-avoidance which has a natural interpretation
in terms of both monotone triangles (columnwise difference less than
two) and in terms of ASMs (natural bijection with TSSCPPs). However, a
number of questions remain unanswered.

First, we would like to understand better how the number of monotone
triangles in bijection with the subset of magogs grows with the
size. We have shown that ASMs grow superexponentially in size
\eqref{asmasym} (see also \cite{blefo}), with entropy $\lambda_1
\approx 1.29904$ \eqref{asmentr}. We have been
able to show that the number of ASMs in bijection grow like
\eqref{pbranasym}, with entropy $\lambda_2 \approx 1.13975$
\eqref{pbranentr}.  For some strange reason, we find that
$\lambda_1=\lambda_2^2$.  What we conjecture is that the entropy is
actually, using \eqref{angrowth}, given by $\ds
\frac{3^{9/8}}{2^{3/2}} \approx 1.21679$.

Secondly, we observe that we do not know if there is a meaningful
interpretation of this bijection on the TSSCPP side. If one believes
that this bijection is not artificial, then one would expect some
structure on TSSCPPs. We have not been able to find this structure so
far.

Finally, one would be interested in exploiting other
pattern avoidances on Gog words. The definition of 312-subpatterns
uses a notion of activity that seems somehow specific to the
subpattern 312. One would like a more holistic definition, which
encompasses other subpatterns.

\section*{Acknowledgements}
We thanks P. Duchon, O. Guibert and X. Viennot for
discussions, Florent Le Gac for collaborating on some of the results
in Section~\ref{sec:rect}, and J. Propp for informing us about the work of
H. Canary.

\appendix

\section{Gapless monotone triangles} \label{sec:ngaplessasm}
A lower bound for the  number of gapless monotone triangles is given
by  the sequence $\alpha_n$ defined by:

\be 
\alpha_1 = 1, \alpha_2 = 2, \alpha_3 = 6, \alpha_4 = 26, \,
\alpha_n  =
\rho_{\lceil\frac{n}{2}\rceil,\lfloor\frac{n}{2}\rfloor}\,
\alpha_{\lceil\frac{n}{2}\rceil},
\ee 
\noindent
where $\rho_{m,p}$ is the number of rectangular gapless shapes with
$p$ rows and $m$ columns given by
\be
\rho_{m,p}\ \ =
\ \ \prod_{i=1}^{p-1}\left(\frac{m+i}{i}\right)^{\lceil\frac{i}{2}\rceil}
\prod_{i=1}^{p}\left(\frac{m+2p-i}{2p-i}\right)^{\lceil\frac{i}{2}\rceil}.
\ee 

In order to determine the asymptotic behaviour of 
$\alpha_n$ it  is convenient to start with   the  formula   
giving $\rho_{2n,2n}$ which is
\be
\rho_{2n,2n}\ \ = \ \ \prod_{i=1}^{2n-1} \left( \frac{2n+i}{i} 
\right)^{\lceil\frac{i}{2}\rceil}
\prod_{i=1}^{2n}\left(\frac{6n-i}{4n-i}\right)^{\lceil\frac{i}{2}\rceil}.
\ee
Separating the case  $i=2j$ from  $i=2j-1$ we get:
\be
\begin{split}
\rho_{2n,2n}\ \ =  \prod_{j=1}^{n-1} &\left(\frac{2n+2j}{2j}\right)^{j}
\  \prod_{j=1}^{n}\left(\frac{2n+2j-1}{2j-1}\right)^{j} \\
 & \times\prod_{j=1}^{n}\left(\frac{6n-2j}{4n-2j}\right)^j
\ \prod_{j=1}^{n}\left(\frac{6n-2j+1}{4n-2j+1}\right)^j
\end{split}
\ee
Taking the logarithms we have:
\be
\begin{split}
&\log (\rho_{2n,2n}) =\sum_{1\le j < n} j\log(2n+2j) 
-\sum_{1\le j< n} j\log(2j) \\
&+ \sum_{1\le j \le n}
j\left( \log(2n+2j -1)+\log(6n-2j)+\log(6n-2j+1)\right) \\
& -  \sum_{1\le j \le n}
j\left(\log(2j-1)+\log(4n-2j)+\log(4n-2j+1)\right).
\end{split}
\ee
Let  $f(x)$ be   the function
\be
\begin{split}
f(x) &= x\left( \log(2n+2x)+\log(2n+2x-1)
+\log(6n-2x) \right. \\ 
&\left. +\log(6n-2x+1)\right) 
 -x\left( \log(2x)+\log(2x-1) \right.\\
&\left. +\log(4n-2x) +\log(4n-2x+1)\right),
\end{split}
\ee
which  relates to $\rho_{2n,2n}$ by the formula
\be
\log(\rho_{2n,2n}) \ \ =\ \ \sum_{j=1}^n f(j) - n\log(2) \ \ =\ \  S_n
- n\log(2)
\ee
The next step consists in approximating the  sum  $S_{n} =\sum_{j=1}^n
f(i)$ using the Euler-Maclaurin formula,
\be
S_n= \frac{f(1)+f(n)}{2}+ I_n +\sum_{k=2}^{\infty}\frac{B_k}{k!}\left(
f^{(k-1)}(n)-f^{(k-1)}(1)\right),
\ee
where $I_{n}$ is the integral
 \be I_n \ = \  \int_{1}^{n} f(x) dx\ee
and $B_k$ are the Bernoulli numbers ($B_1=-1/2$, $B_2=1/6$, $B_3=0$, 
$B_4=-1/30$, $\cdots$).

We can evaluate $I_n$ using the formula:
\be 
\int  x\,\log(ax+b) dx \ \ \ = \ \ \ 
\frac{x^2}{2} \log(ax+b) - \frac{x^2}{4} 
 - \frac{bx}{2a}  - \frac{b^2}{2a} \log(ax+b),
\ee
giving
\be
I_n\ = \ A(n)  n^2 + B(n) n + C(n),
\ee
where
\be
\begin{split}
A(n) &=-4\log(2)-3\log(n)+\frac{9}{2}\log(6n-1)
-4\log(4n+1) \\
&-2\log(4n-1)+\frac{9}{2}\log(3n-1)\\
&+ 2\log(2n+1)-\frac{5}{2}\log(2n-1)
+\frac{1}{2}\log(n+1),\\
B(n)&=\frac{3}{2}\log(6n-1)
-\frac{3}{2}\log(4n+1)-\frac{1}{2}\log(4n-1) 
+ \frac{1}{2}\log(2n+1),\\
C(n)&=-\frac{3}{8}\log(6n-1)
-\frac{1}{8}\log(4n+1) 
+\frac{1}{4}\log(4n-1) \\
&-\frac{1}{2}\log(3n-1)
  -\frac{1}{4}\log(2n+1)
+\frac{5}{8}\log(2n-1)-\frac{1}{2}\log(n+1).
\end{split}
\ee
Using the approximation formula
\be
\log(an+b) \ = \ \log(a)+\log(n)+
\frac{b}{an}-\frac{b^2}{2a^2n^2}+O(1/n^3),
\ee  
we obtain 
the following approximations when ${n\rightarrow\infty}$
\be
\begin{split}
I_n+\frac{f(1)+f(n)}{2}-n\log(2))
&= n^2(9\log(3) -12\log(2)) \\
&+ n (\frac{3}{2}\log(3) -\log(2)) +\frac{1}{8}\log(n) +O(1),\\
\frac{B_2}{2!}\left( f'(n)-f'(1)
\right) &=-\frac{1}{6}\log(n)+O(1), \\
 f^{(3)}(1)-f^{(3)}(0) &= O(1),
\end{split}
\ee
giving
\be
\begin{split}
\log(\rho_{2n,2n}) =&  n^2(9\log(3) -12\log(2)) \\
&+ n (\frac{3}{2}\log(3) -\log(2)) -\frac{1}{24}\log(n) +O(1),\\
\rho_{2n,2n}=&\gamma \left( 
\frac{3^{9n^2+\frac{3}{2}n}}{2^{12n^2+n}}
\right)n^{-1/24}\left( 1+O(1/\log(n))\right),
\end{split}
\ee
where  $\gamma$ is a constant. Hence
\begin{eqnarray*}
\rho_{n,n}&=&\gamma\,2^{-1/24}\left( 
\frac{3^{9n^2+3n}}{2^{12n^2+2n}}
\right)^{1/4}n^{-1/24}\left( 1+O(1/\log(n))\right)
\end{eqnarray*}
Since we have
\be\alpha_n=\rho_{n/2,n/2}\alpha_{n/2} ,\ee
we obtain:
\be  \label{pbranasym}
\alpha_n = \beta \left( \frac{3^{3/4}}{2}
\right)^{n^2}\left( \frac{3^{3/4}}{2^{1/2}}\right)^{n}
n^{-1/24}\left( 1+O(1/\log(n))\right) 
\ee
where  $\beta$ is a constant.

\section{Alternating signs matrices} \label{sec:nasm}

Bleher and Fokin \cite{blefo} have already given an asymptotic
estimate of the number $u_n$ of ASMs using Riemann-Hilbert methods. 
For completeness, we give an estimate using more elementary
techniques. Recall that
\be 
u_n=\prod_{0\le i<n}\frac{(3i+1)!}{(n+i)!}
\ee
We proceed as above taking the logarithms,
\be
\begin{split}
\log (u_n)&=\sum_{1\le i<n} 
(n-i)\left( \log(3i-1)+\log(3i)+\log(3i+1)\right) \\
&-\sum_{1\le i<n} \left(n \log(i+1)+(n-i) \log(n+i)\right).
\end{split}
\ee
Then we consider the function:
\be
\begin{split}
f(x) &= (n-x)\left( \log(3x-1)+\log(3x) \right. \\
& \left. +\log(3x+1)-\log(n+x)\right) -n \log(x+1),
\end{split}
\ee
and the integral
\begin{eqnarray*}
I_{n} &=&\int_{1}^{n-1} f(x) dx.
\end{eqnarray*}
This time we get
\be
\begin{split}
I_n&=\frac{4}{3}\,(1-n)\log(2)+\frac{1}{2}\,n(n-2)\log(3)-n^2 \log(n) \\
&+\left( \frac{(3n+2)(3n-4)\log(3n-4)+(3n+4)(3n-2)\log(3n-2)}{18}\right) \\
& +\left( \frac{(n-1)(n+1)\log(n-1)}2 \right. \\
&\left. \frac{+(n+1)(3n-1)\log(n+1)-(2n-1)(2n+1)\log(2n-1)}{2}\right).
\end{split}
\ee
Again the Euler-Maclaurin formula gives
\be
\begin{split}
S_n-\frac{1}{2}&(f(1)+f(n-1))=\frac{f(1)}{2}+f(2)\cdots
+f(n-2)+\frac{f(n-1)}{2} \\ 
&= I_n +\sum_{k=1}^{\infty}\frac{B_{2k}}{(2k)!}\left(
f^{(2k-1)}(n-1)-f^{(2k-1)}(1)\right).
\end{split}
\ee
%Where the  $B_k$ are the Bernoulli numbers;
Using $\log(an+b) =  \log(a)+\log(n)+\frac{b}{an}
-\frac{b^2}{2a^2n^2}+O(1/n^3)$, we get
\be
\begin{split}
 I_n+\frac{1}{2}(f(1)+f(n))
&=\frac{1}{9}\log(n)-n\left( 2n+\frac{1}{3}\right)\log(2) \\
&+n\left( \frac{3}{2}n-\frac{1}{2}\right)\log(3)+n +O(1), \\
\frac{B_2}{2!}\left( f'(n-1)-f'(1)
\right)&=-\frac{11}{48}n-\frac{1}{4}\log(n)+O(1), \\
\frac{B_4}{4!}\left( f^{(3)}(n-1)-f^{(3)}(1)\right)&=
\frac{299}{23040}n+O(1), \\
\log(u_n) &=-2n^2\log(2)+\frac{3}{2}\,n^2\log(3)\\
&+n\,\log(\beta )-\frac{5}{36}\log(n)+O(1),
\end{split}
\ee
which leads to
\be \label{asmasym}
u_n=\gamma\,\,\left( \frac{27}{16}
\right)^{n^2/2}\beta^n\,\,n^{-5/36}\left( 1+O(1/\log(n))\right),
\ee
where $\gamma$ and $\beta$ are constants.
Notice that the number of ASMs grows like 
$(\lambda_1)^{n^2}$ where
\be \label{asmentr}
\lambda_1= \left(\frac{27}{16}\right)^{1/2}=\frac{3 \sqrt{3} }4,
\ee
while we have found the lower bound $(\lambda_2)^{n^2}$ for the number of
those avoiding $312$ is
\be \label{pbranentr}
\lambda_2 = \frac{3^{3/4}}{2}.
\ee

\bibliographystyle{alpha.bst}
\bibliography{asm}

\end{document}